\documentclass[11pt]{amsart}
\usepackage{epsfig}
\usepackage{amssymb}
\usepackage{amsthm}
\usepackage{amsmath}
\usepackage{rotating}
\usepackage{eucal}
\usepackage{enumitem} 
\usepackage{url,hyperref}

\def\tldC{{\tld{C}}}
\def\Lns{L_n^{(\mathbf{s})}}

\def\phi{\varphi}       
\def\tld{\widetilde}     
\def\row{\rightarrow}

\def\lam{\lambda}
\newcommand{\la}{\lambda}

\newcommand{\floor}[1]{\ensuremath{\left\lfloor #1 \right\rfloor}}
\newcommand{\ang}[1]{\ensuremath{\left\langle #1 \right\rangle}}

\newcommand{\ceil}[1]{\ensuremath{\left\lceil #1 \right\rceil}}

\newcommand{\abs}[1]{\ensuremath{\left\lvert #1 \right\rvert}}

\newcommand{\qbinom}[2]{\genfrac{[}{]}{0pt}{}{#1}{#2}}

\def\inv{\mathsf{inv}}
\def\neg{\mathsf{neg}}
\renewcommand\max{\mathsf{max}}

\def\lhp{\mathsf{lhp}}
\def\comaj{\mathsf{comaj}}
\def\amaj{\mathsf{amaj}}
\def\ades{\mathsf{ades}}
\def\Asc{\mathsf{Asc}}
\def\asc{\mathsf{asc}}
\def\Des{\mathsf{Des}}
\def\des{\mathsf{des}}
\def\inv{\mathsf{inv}}
\def\nsp{\mathsf{nsp}}
\def\last{\mathsf{last}}

\newcommand\bbZ{\mathbb{Z}}
\newcommand{\integers}{\mathbb{Z}}
\newcommand{\I}{\mathbf{I}}

\newcommand\s{\mathbf{s}}
\newcommand\id{\mathrm{id}}

\theoremstyle{plain}
\newtheorem{theorem}{Theorem}[section]
\newtheorem{lemma}[theorem]{Lemma}
\newtheorem{corollary}[theorem]{Corollary}

\newtheorem{observation}[theorem]{Observation}

\theoremstyle{definition}

\newtheorem{example}[theorem]{Example}

\theoremstyle{remark}

\numberwithin{equation}{section}

\usepackage{tikz}

\title[Lecture hall partitions and the affine hyperoctahedral group]{Lecture hall partitions and\\ the affine hyperoctahedral group}

\author{Christopher R.\ H.\ Hanusa}
\address{Department of Mathematics \\ Queens College (CUNY) \\ 65-30 Kissena Blvd. \\ Flushing, NY 11367\\ United States}
\email{\href{mailto:chanusa@qc.cuny.edu}{\texttt{chanusa@qc.cuny.edu}}}
\urladdr{\url{http://qc.edu/~chanusa/}}

\author{Carla D.\ Savage}
\address{Department of Computer Science \\ North Carolina State University \\ Raleigh, NC 27695-8206 \\ United States}
\email{\href{mailto:savage@ncsu.edu}{\texttt{savage@ncsu.edu}}}
\urladdr{\url{http://www4.ncsu.edu/~savage/}}



\begin{document}

\thispagestyle{empty}

\begin{abstract}
In 1997 Bousquet-M\'elou and Eriksson introduced lecture hall partitions as the inversion vectors of elements of the parabolic quotient $\widetilde{C}/C$.  We provide a new view of their correspondence that allows results in one domain to be translated into the other.  We determine the equivalence between combinatorial statistics in each domain and use this correspondence to translate certain generating function formulas on lecture hall partitions to new observations about $\widetilde{C}/C$.
\end{abstract}

\maketitle
\thispagestyle{empty}

\vspace{-.3in}

\section{Introduction}
In 1997  Bousquet-M\'elou  and Eriksson  \cite{BME}  introduced lecture hall partitions by showing that they are the inversion vectors of elements of the parabolic quotient $\tldC/C$. We provide a new way to understand their correspondence which allows results in one domain to be translated into the other.  In Section~\ref{sec:proof} we distill the essence of the new proof into a correspondence between a $(2,4,\hdots,2n)$-inversion sequence (the excess of the lecture hall partition) and a signed permutation (the residues modulo $2n+2$ of the window of the element of the parabolic quotient).  Indeed, a key insight is that the set of lecture hall partitions $L_n$ is the same as the set of generalized \mbox{$\s$-lecture} hall partitions for $\s=(2,4,\hdots,2n)$; this appears to be a more natural environment for understanding the correspondence with the affine Coxeter group $\tldC_n$.   

The structure of this paper is as follows.  In this section we introduce just enough notation and background to state the result of Bousquet-M\'elou  and Eriksson.  In Section~\ref{sec:proof} we give our new proof of their result. In Section~\ref{sec:stats}, we show how this new correspondence allows us to find the equivalence of combinatorial statistics in the two regimes.  This leads to Section~\ref{sec:results} in which we present new proofs and novel results in $\tldC/C$ corresponding to truncated lecture hall partitions, as well as a product-form  generating function for a quadratic statistic $\lhp_C$ on $C_n$.  Bj\"orner and Brenti  \cite{BB} is our reference on Coxeter groups.

\subsection{The hyperoctahedral group}\

\noindent
The {\em hyperoctahedral group} $C_n$ (or $B_n$) is the (finite) group of symmetries of the $n$-cube, which can be defined as a reflection group with generators $\{s_0,\hdots,s_{n-1}\}$ and the relations $s_i^2=\id$ for $0\leq i\leq n-1$, $(s_is_{i+1})^3=\id$ for $1\leq i\leq n-2$, and $(s_0s_1)^4=\id$; all other pairs of generators commute.  $C_n$ is also seen as the set of permutations $\sigma$ of \[\{-n,-(n-1),\hdots,n-1,n\}\] satisfying 
$\sigma(-i)=-\sigma(i)$ for $1\leq i\leq n$.  These permutations are completely defined by the values of the {\em window} $[\sigma(1), \hdots, \sigma(n)]$, which is a permutation of $\{\pm 1,\hdots,\pm n\}$ in which exactly one of $+i$ or $-i$ appears for $1\leq i\leq n$. In this way $C_n$ is realized as the set of {\em signed permutations}.

\subsection{The affine hyperoctahedral group}\label{sec:tldC}\

\noindent
The {\em affine hyperoctahedral group} $\tldC_n$ is an (infinite) reflection group that includes all the generators and relations of $C_n$ and along with one additional generator, $s_n$, satisfying $s_n^2=\id$ and one more non-commuting relation, $(s_{n-1}s_n)^4=\id$.  Elements $w$ of $\tldC_n$ can be written as permutations of $\bbZ$ satisfying 
\begin{equation}\label{eq:transC1}
w_{i+(2n+2)}=w_{i}+(2n+2)
\end{equation}
 and 
\begin{equation}\label{eq:transC2}
w_{-i}=-w_i
\end{equation}
for all $i\in \bbZ$ \cite{Henrik, EE}.  These conditions imply $w(i)=i$ for all $i\equiv 0\bmod (n+1)$. We will write $N=2n+2$.

As above, an element $w\in\tldC_n$ is completely defined by the window $[w_1,\hdots,w_n]$.  We know that a window corresponds to an element of $\tldC_n$ if the values $\pm w_1, \hdots, \pm w_n$ are all distinct modulo $N$.

The group $C_n$ embeds as a parabolic subgroup of $\tldC_n$; as such each element $w\in \tldC_n$ has a parabolic decomposition $w=w^0w_0$ where $w^0\in\tldC_n/C_n$ and $w_0\in C_n$.  Bj\"orner and Brenti \cite[Proposition~8.4.4]{BB}, show that the entries of a window of an element ${w} \in \tldC_n/C_n$ are all positive and sorted in increasing order.   
 
\subsection{Class inversions}\

\noindent
Analogous to the concept of inversions for permutations is the idea of class inversions for $w \in \tldC_n$.
For $j \in \bbZ$, define the {\em class} $\ang{j}$ to be the set of all positions of a form $j+kN$ or $-j+kN$ for all $k\in \bbZ$.  We will keep track of the number of {\em class inversions}, which are entries in positions of the class $\ang{j}$ that are inverted with a fixed entry in position $i$.  More precisely, for $1 \leq j \leq i \leq n$ define 
\begin{equation}
\label{eq:Iij}
\begin{aligned}
I_{i,j} &=  \abs{\{k\in\bbZ_{> 0} : w_i>w_j+kN\}}+ \abs{\{k\in\bbZ_{> 0} : w_i>-w_{j}+kN\}}\\
&=\floor{\frac{w_i - w_j}{N}}  +  \floor{\frac{w_i + w_j}{N}}.\\
\end{aligned}
\end{equation}
For $1\leq i\leq n$, let $I_i = \sum_{1 \leq j \leq i} I_{i,j}$. Define the {\em class inversion vector} $I(w)$ of $w$ as
\[
I(w) = [I_1, I_2, \ldots, I_n].
\]
The {\em type-C affine inversion number} of $w \in \tldC_n$ is given by
\[
\inv_{\tldC}(w) = \sum_{i=1}^n I_i,
\]
which is equal to the Coxeter length $\ell(w)$ of $w$ \cite[Proposition~8.4.1]{BB}.

\subsection{Lecture hall partitions and class inversion vectors}\

\noindent
Lecture hall partitions, introduced by Bousquet-M\'elou and Eriksson in \cite{BME},  arose naturally in the study of  inversions in  $\tldC_n/C_n$.

The set $L_n$ of {\em lecture hall partitions} of length $n$ is the set of integer sequences
\[
L_n = \left\{(\lam_1,\hdots,\lam_n) \in \integers^n: 0\leq \frac{\lam_1}{1} \leq \frac{\lam_2}{2}\leq \cdots \leq \frac{\lam_n}{n}\right\}.
\]

We will focus on the following characterization of lecture hall partitions and its implications.
\begin{theorem}[Bousquet-M\'elou and Eriksson \cite{BME}]
The mapping sending 
 $w \in \tldC_n/C_n$ to its class inversion vector $[I_1, \ldots, I_n]$ is a bijection
 $$\tldC_n/C_n\row L_n.$$
 \label{CtoL}
\end{theorem}

In the next section we provide a new view of this bijection that allows results in one domain to be translated into the other.

\newpage
\section{A new view of the bijection}\label{sec:proof}

In this section, we will give a proof of Theorem \ref{CtoL} with the following overview.

\begin{itemize}
   
\item Show that the elements $\la \in L_n$ can be characterized as pairs in this set:
\[
   T_n^{(\s)} = \left \{(b,e) \in \integers^n \times   \I_n^{(\s)} :  {\rm for}  \ i \in \{0, \ldots, n-1\}, \  b_i \leq b_{i+1} \ {\rm and} \ i \in \Asc^{(\s)}(e) \implies b_i < b_{i+1}  \right \},
\]
for the sequence $\s=(2,4,\ldots, 2n)$, in such a way that 
$$\abs{\la} = \sum_{i=1}^n 2ib_i- \abs{e}.$$
Inversion sequences  $\I_n^{(\s)}$ and the ascent statistic $\Asc^{(\s)}$ are defined in Section \ref{genlhp}.
 
\item Show that elements $w \in \tldC_n/C_n$ can be characterized as pairs in this set:
\[
   U_n = \left \{(c,\sigma) \in \integers^n \times  C_n : {\rm for}  \ i \in \{0, \ldots, n-1\}, \  c_i \leq c_{i+1} \ {\rm and} \ 	i \in \Des(\sigma) \implies c_i < c_{i+1}  \right \}.
\]
   in such a way that $$\inv_{\tldC}(w)= \sum_{i=1}^n 2ic_i- \inv_C(\sigma).$$
   
\item Show  a bijection $\Psi:  C_n \rightarrow \I_n^{(\s)}$ with the property that for $\sigma \in C_n$,
\[
\Des(\sigma) = \Asc^{(\s)}(\Psi(\sigma))\ \  \ {\rm  and} \ \ \  \inv_C(\sigma) = \abs{\Psi(\sigma))}=\abs{e}.
\]
   
\item  Conclude that the following action sends $w \in \tldC_n/C_n$ to its inversion vector $I(w)=\la$:
\[
   w \longleftrightarrow (c, \sigma) \longleftrightarrow (c,\Psi(\sigma)) \longleftrightarrow \la
\]
\end{itemize}

\subsection{Generalized lecture hall partitions}\label{genlhp}\

\noindent
A natural generalization of $L_n$ is to consider the set $\Lns$ of {\em $\s$-lecture hall partitions}, defined for a sequence $\s=(s_1,\hdots,s_n)$ of positive integers by
\[
\Lns = \left\{(\lam_1,\hdots,\lam_n): 0\leq \frac{\lam_1}{s_1} \leq \frac{\lam_2}{s_2}\leq \cdots \leq \frac{\lam_n}{s_n}\right\}.
\]
 In this paper we will be exploiting the fact that the sets $L_n$ and $L_n^{(2,4,\hdots,2n)}$ are equal.
 
 For $\lam \in \Lns$, define the ceiling of $\lam$ with respect to $\s$ by
 \[
 \ceil{\lam}^{(\s)}:=\left(
\Big\lceil\! \frac{\lam_1}{s_1}\!\Big\rceil, 
\Big\lceil\! \frac{\lam_2}{s_2}\!\Big\rceil, 
\hdots, 
\Big\lceil\! \frac{\lam_n}{s_n}\!\Big\rceil\right)\]
and the {\em excess} of $\la$ with respect to $\s$ by
\[
e^{(\s)}(\lam) := \left(
s_1\Big\lceil\! \frac{\lam_1}{s_1}\!\Big\rceil-\lam_1, 
s_2\Big\lceil\! \frac{\lam_2}{s_2}\!\Big\rceil-\lam_2, 
\hdots, 
s_n\Big\lceil\! \frac{\lam_n}{s_n}\!\Big\rceil-\lam_n\right).
\]
 
 \begin{observation}
 \label{lhpchar}
 For any sequence $\s$ of positive integers, all of the following are consequences of $\lam \in \Lns$.
 \begin{enumerate}[label=(\alph*)]
 \item
 $\ceil{\lam}^{(\s)}$ is a nondecreasing sequence of positive integers.
 \item
$0 \leq e_i^{(\s)}(\la) < s_i$  for $1 \leq i \leq n$.
 \item
 If $\ceil{\lam}_1^{(\s)}=0$ then $e_1^{(\s)}(\la)=0$.
 \item
 For $1 \leq i < n$,
 \[
 \text{If }\ceil{\lam}_i^{(\s)} =  \ceil{\lam}_{i+1}^{(\s)}  \text{ then } \frac{e_i^{(\s)}(\la)}{s_i} \geq  \frac{e_{i+1}^{(\s)}(\la)}{s_{i+1}}.
 \]
 \end{enumerate}
   \end{observation}
   
   Observation \ref{lhpchar} motivated the following definitions from \cite{SS}.
   
   For a sequence $\s=(s_1, \ldots, s_n)$ of positive integers,  the  {\em $\s$-inversion sequences} are defined by
   \[
   \I_n^{(\s)} = \{ (e_1, \ldots, e_n) \in \integers^n : 0 \leq e_i < s_i, \ 1 \leq i \leq n \}.
   \]
   For $e \in    \I_n^{(\s)}$, the {\em ascent set} of $e$ with respect to $\s$ is 
   \begin{equation}\label{eq:ascs}
   \Asc^{(\s)}(e) = \left \{ i \in \{0, \ldots, n-1\} \ :\ \frac{e_i}{s_i} < \frac{e_{i+1}}{s_{i+1}}  \right \},
   \end{equation}
   where  for the convenience of the definition we let $e_0=0$ and $s_0=1$.

   If $\lam \in \Lns$,  Observation \ref{lhpchar}(b) says that $e^{(\s)}(\lam) \in  \I_n^{(\s)}$ and 
   Observations \ref{lhpchar}(a,c,d) imply that if $i \in \Asc^{(\s)}(e)$ then  $\ceil{\lam}_i^{(\s)} <  \ceil{\lam}_{i+1}^{(\s)}$ and otherwise $\ceil{\lam}_i^{(\s)} \leq  \ceil{\lam}_{i+1}^{(\s)}$.

\begin{example}   
   The sequence $(1,2,1,7,8)$ is a $(2,4,6,8,10)$-inversion sequence and its ascent set is $\Asc^{(2,4,6,8,10)}(1,2,1,7,8)= \{0,3\}$.  We note that position 1 is not a $(2,4,6,8,10)$-ascent since $\frac{1}{2} \not < \frac{2}{4}$, while
   4 is not a $(2,4,6,8,10)$-ascent because $\frac{7}{8} \not < \frac{8}{10}$
\end{example}

It was shown in \cite{SS} that these conditions characterize $ \Lns$.
 
   \begin{lemma}[\cite{SS}]
   For any sequence $\s$ of positive integers, the mapping
   \[
   \la \longrightarrow \big(\ceil{\lam}^{(\s)}, e^{(\s)}(\lam)\big)
   \] 
   is a bijection
   \[
   \Lns \rightarrow T_n^{(\s)}
   \]
   where, under the convention that $b_0=0$,
   \[
   T_n^{(\s)} = \left \{(b,e) \in \integers^n \times   \I_n^{(\s)} :  {\rm for}  \ i \in \{0, \ldots, n-1\}, \  b_i \leq b_{i+1} \ {\rm and} \ i \in \Asc^{(\s)}(e) \implies b_i < b_{i+1}  \right \}.
   \]
   \end{lemma}
   
\subsection{A characterization of $\tldC_n/C_n$}\label{Cchar}\

\noindent
From Section \ref{sec:tldC}, each representative window  $w =  [w(1), \ldots, w(n)]  \in \tldC_n/C_n$ is an increasing sequence of positive integers with a unique representation
 \[
 [w(1), \ldots, w(n)] = [c_1 N + \sigma_1, c_2N + \sigma_2, \ldots, c_nN+\sigma_n]
 \]
 where $\sigma=(\sigma_1, \ldots, \sigma_n) \in C_n$.  A {\em descent} of $\sigma \in C_n$ is a position $i \in [n-1]$ such that $\sigma_i > \sigma_{i+1}$.  $\Des(\sigma)$ is the set of all descents of $\sigma$ and $\des(\sigma) = \abs{\Des(\sigma)}$.
 \begin{observation}
 \label{affCchar}
 All of the following properties are consequences of the definitions:
 \begin{enumerate}[label=(\alph*)]
 \item
 $0 \leq c_1 \leq c_2 \leq \ldots \leq c_n$.
 \item
 If $c_1=0$ then $\sigma_1 >0$.
 \item
 For $1 \leq i < n$ if $i \in \Des(\sigma)$ then $c_i < c_{i+1}$

 \end{enumerate}
   \end{observation}
   
We show these conditions characterize $\tldC_n/C_n$.
 
   \begin{lemma}
   The mapping
   \[
   w=[w_1, \ldots, w_n]   \longrightarrow (c,\sigma),
   \] 
   where $w_i = c_iN+\sigma_i$ and $\sigma \in C_n$
   is a bijection
   \[
   \tldC_n/C_n \rightarrow U_n,
   \]
   where, under the convention that $c_0=0$,
   \[
   U_n = \left \{(c,\sigma) \in \integers^n \times  C_n : {\rm for}  \ i \in \{0, \ldots, n-1\}, \  c_i \leq c_{i+1} \ {\rm and} \ i \in \Des(\sigma) \implies c_i < c_{i+1}  \right \}.
   \]
   \end{lemma}
   \begin{proof}
   Let $(c,\sigma)$ be an element of $U_n$.  We need to verify that $[c_1N+\sigma_1, \ldots, c_nN+\sigma_n]$ is the window of a coset representative in $\tldC_n/C_n$.  Since the residues mod $N$ are distinct and not congruent to $0$ or $n+1$ we need only verify that the window elements are positive and in weakly increasing order.
   
   To check that $c_1N+\sigma_1 > 0$, this is true if $\sigma_1>0$.  But if $\sigma_1 < 0$, then position 0 is a descent of $\sigma$ and, by definition of $U_n$, $c_0 = 0 < c_1$.  So, $c_1$ is positive and therefore $c_1N+\sigma_1 > 0$.
   
   For $1 \leq i < n$, to check that $c_iN+\sigma_i  \leq c_{i+1}N + \sigma_{i+1}$, this is clearly true if $c_i<c_{i+1}$.  If not, by definition of $U_n$, it  must be that $c_i=c_{i+1}$ and therefore that $i \not \in \Des(\sigma)$.  Thus $\sigma_i < \sigma_{i+1}$ and therefore
   $c_iN+\sigma_i  \leq c_{i+1}N + \sigma_{i+1}$.
   \end{proof}

\subsection{Inversion sequences, permutations, and signed permutations}\
    
\noindent  
For a permutation $\pi=(\pi_1, \ldots,  \pi_n) \in S_n$, an {\em inversion} of $\pi$ is a pair $i < j$ where $\pi_i > \pi_j$ and a {\em descent} of $\pi$ is a position $i \in [n-1]$ such that $\pi_i > \pi_{i+1}$. The number of inversions of $\pi$ is denoted $\inv(\pi)$ while $\Des(\pi)$ is the set of all descents of $\pi$ and $\des(\pi) = \abs{\Des(\pi)}$.

For a signed permutation $\sigma=(\sigma_1,\ldots, \sigma_n) \in C_n$, there are (at least) two notions of inversion number. One standard definition is
\[
\inv(\sigma) = \big\{(j,i) : 1 \leq j < i \leq n  \ {\rm and} \  \sigma_j > \sigma_i\big\}.
\]
A more natural inversion number from \cite{BB} aligns with the Coxeter length of $\sigma$; for this we also need the number of {\em negative sum pairs},
\[
\nsp(\sigma)  = \big\{(j,i) : 1 \leq j < i \leq n  \ {\rm and} \  \sigma_j + \sigma_i<0\big\}
\]
and the number of negative entries 
\[
\neg(\sigma) = \#\{i \in [n] : \sigma_i < 0\}.
\]
The inversion number for $\sigma \in C_n$ is then defined by
\[
\inv_C(\sigma) = \inv(\sigma) + \neg(\sigma) + \nsp(\sigma).
\]

Classical inversion sequences $e\in\I_n= \I_n^{(1,2, \ldots, n)}$ have been used in various ways to encode permutations.  Define 
$$\Theta: S_n \rightarrow \I_n$$
       by  $\Theta(\pi) = (e_1 \ldots, e_n)$ where
       \[
e_i = \#\{j \in [i-1] : \pi_j > \pi_i\}.
\]
Clearly $\inv(\pi) = \abs{\Theta(\pi)}$.  Moreover, it was shown in \cite{SS} that
\[
\Des(\pi) = \Asc(\Theta(\pi)).
\]

On the other hand, $(2,4,\hdots,2n)$-inversion sequences can be used to encode signed permutations.  
Define 
\[
\Psi: C_n \rightarrow \I_n^{(2,4, \ldots, 2n)}
\]
as follows.
For the signed permutation $\sigma \in C_n$, create the unsigned permutation $\abs\sigma=(\abs{\sigma_1}, \ldots, \abs{\sigma_n})$ and let
$e^*=\Theta(\abs\sigma)$.  This means $e^*_i$ is the number of $j \in [i-1]$ such that
$\abs{\sigma_j} > \abs{\sigma_i}$.
Now define $\Psi(\sigma)=e=(e_1, \ldots, e_n)$ where
\[
e_i = \left \{
\begin{array}{ll}
e_i^* & {\rm if} \  \sigma_i > 0\\
2i-1-e_i^* & {\rm otherwise.}
\end{array}
\right .
\]
Lemma~\ref{lem:invsigma} will prove that $\inv_C(\sigma)  =  \abs{\Psi(\sigma)}$ and \cite[Theorem 3.12\,(4)]{SV} proves
\[\Des(\sigma)=\Asc^{(2,4,\hdots,2n)}\big(\Psi(\sigma)\big).\]
We remind the reader that this notion of ascent set from Equation~\eqref{eq:ascs} is more subtle than the classic notion of ascent set.

\begin{example}
For the signed permutation $\sigma=(-3,-1,2,-5,-4) \in C_5$, the unsigned permutation $\abs\sigma$ is $(3,1,2,5,4)$, whose inversion sequence $e^*=(0,1,1,0,1)$.  As a consequence, \[\Psi(\sigma) = (1,2,1,7,8) \in \I_5^{(2,4,6,8,10)}.\] In addition, $\inv_C(\sigma)=19=1+2+1+7+8$ and 
\[\Des(-3,-1,2,-5,-4) = \{0,3\} = \Asc^{(2,4,6,8,10)}(1,2,1,7,8).\]   
\end{example}

For a Boolean function $f$, let $\chi(f)=1$ if $f$ is true and 0 otherwise.  The following lemma is straightforward.

\begin{lemma}
\label{eij}
For $\sigma \in C_n$,
\[
\chi(\abs{\sigma_j} > \abs{\sigma_i}) = 
\left \{
\begin{array}{ll}
\chi(\sigma_j > \sigma_i) + \chi(\sigma_j +\sigma_i < 0) & {\rm if} \ \sigma_i > 0\\
2-\chi(\sigma_j > \sigma_i) - \chi(\sigma_j +\sigma_i < 0) & {\rm if} \ \sigma_i < 0\\
\end{array}
\right .
\]
\end{lemma}

\begin{corollary}
\label{eidef}
Given $\sigma \in C_n$, define $e=\Psi(\sigma)=(e_1, \ldots, e_n)$.  Then for $1 \leq i \leq n$,
\[
e_i = \sum_{j=1}^{i} \big(\chi(\sigma_j > \sigma_i) + \chi(\sigma_j +\sigma_i < 0)\big).
\]
\end{corollary}
\begin{proof}
By definition of $\Psi$, if $\sigma_i > 0$,
\[
e_i = \sum_{j=1}^{i-1} \chi(\abs{\sigma_j} > \abs{\sigma_i}).
\]
Since   $\chi(\sigma_i>\sigma_i) = 0$ and  $\chi(\sigma_i+\sigma_i<0) = 0$, the result follows from Lemma \ref{eij}.

If $\sigma_i < 0$, then by definition of $\Psi$,
\begin{eqnarray*}
e_i &= & 2i-1 - \sum_{j=1}^{i-1} \chi(\abs{\sigma_j} > \abs{\sigma_i})\\
& = & 2i-1 - \sum_{j=1}^{i-1} \big(2-\chi(\sigma_j > \sigma_i) - \chi(\sigma_j +\sigma_i < 0)\big)\\
& = & 1 + \sum_{j=1}^{i-1} \big(\chi(\sigma_j > \sigma_i)  +\chi(\sigma_j +\sigma_i < 0)\big)
\end{eqnarray*}
and since   $\chi(\sigma_i>\sigma_i) = 0$ and $\chi(\sigma_i+\sigma_i<0) = 1$, the result follows from Lemma \ref{eij}.
\end{proof}

\begin{lemma} \label{lem:invsigma}
$\Psi: C_n \rightarrow \I_n^{(2,4, \ldots, 2n)}$ satisfies, for $\sigma \in C_n$,  
\begin{eqnarray*}
\inv_C(\sigma) & = & \abs{\Psi(\sigma)}
\end{eqnarray*}
\end{lemma}
\begin{proof}
For  $\sigma \in C_n$ with $\Psi(\sigma)=e=(e_1, \ldots, e_n)$, applying Corollary \ref{eidef} gives:
\begin{eqnarray*}
 \abs{\Psi(\sigma)}=\abs{e}& =& \sum_{i=1}^n e_i\\
 & = & \sum_{i=1}^n \sum_{j=1}^{i} \big(\chi(\sigma_j > \sigma_i) + \chi(\sigma_j +\sigma_i < 0)\big)\\
 & = & \sum_{i=1}^n \big(\chi(\sigma_i> \sigma_i) + \chi(\sigma_i+\sigma_i < 0)\big) + \sum_{1 \leq j < i \leq n}\chi(\sigma_j> \sigma_i) + 
 \sum_{1 \leq j < i \leq n}\chi(\sigma_j+\sigma_i <0) \\
 & = & \neg(\sigma) + \inv(\sigma) + \nsp(\sigma) = \inv_C(\sigma).
 \end{eqnarray*}
 \end{proof}

We can now write a new expression to count class inversions of elements in the parabolic quotient $\tldC_n/C_n$ along with a new expression for entries of the class inversion vector.
   
\begin{lemma} 
For $w \in \tldC_n/C_n$ with window
 $[w_1, \ldots, w_n] = [c_1 N + \sigma_1, c_2N + \sigma_2, \ldots, c_nN+\sigma_n]$
where $\sigma=(\sigma_1, \ldots, \sigma_n) \in C_n$ and for $1 \leq j \leq i$,
\[
   I_{i,j}(w) = 2c_i - \chi(\sigma_j > \sigma_i) - \chi(\sigma_j + \sigma_i < 0).
\]
\end{lemma}

\begin{proof}
  From \eqref{eq:Iij},
   \begin{eqnarray*}
    I_{i,j}(w) & = &  \floor{\frac{{w}_i - {w}_j}{N}}  +  \floor{\frac{{w}_i + {w}_j}{N}}\\
    & = &  \floor{\frac{c_iN+\sigma_i - c_jN-\sigma_j}{N}}  +  \floor{\frac{c_iN+\sigma_i +c_jN+\sigma_j}{N}}\\
    & = & 2c_i +  \floor{\frac{\sigma_i - \sigma_j}{N}} +  \floor{\frac{\sigma_i +\sigma_j}{N}}\\
    &  = & 2c_i - \chi(\sigma_j > \sigma_i) - \chi(\sigma_i + \sigma_j<0),
    \end{eqnarray*}
    since $\abs{\sigma_i}+\abs{\sigma_j} < N$.
\end{proof}

\begin{corollary}
For  $w= [c_1 N + \sigma_1, c_2N + \sigma_2, \ldots, c_nN+\sigma_n] \in \tldC_n/C_n$,
   \[
   \inv_{\tldC}(w) = \sum_{i=1}^n 2ic_i - \inv_C(\sigma).
   \]
\end{corollary}

\subsection{Proof of Theorem \ref{CtoL}}
   
\noindent
Collecting the observations from the previous three subsections we have the following.
\begin{theorem}
   \label{newCtoL}
   The mapping 
   \[
   w=[c_1N+\sigma_1,  c_2N+\sigma_2, \ldots, c_nN+\sigma_n] \mapsto (2c_1-e_1, 4c_2-e_2, \ldots, 2nc_n-e_n) = \la,
   \]
   where $w \in  \tldC_n/C_n  $, $(\sigma_1, \ldots, \sigma_n)=\sigma  \in C_n$, and $(e_1, \ldots, e_n) = \Psi(\sigma)$, is a bijection
   $$    \tldC_n/C_n  \rightarrow L_n$$
   satisfying  $\abs{\la} = \inv_{\tldC}(w)$ with $\la_i = I_i(w)$ for $1 \leq i \leq n$.
\end{theorem}
     From this characterization, we know this is exactly the bijection of  Bousquet-M\'elou and Eriksson.

\newpage
\section{A thesaurus of statistics}\label{sec:stats}

In this section, we present the thesaurus between combinatorial statistics on an element 
\[w=[w_1,w_2,\hdots,w_n]=[c_1N+\sigma_1,  c_2N+\sigma_2, \ldots, c_nN+\sigma_n]\in \tldC_n/C_n\] 
and the corresponding lecture hall partition 
\[\la=(\la_1,\la_2,\hdots,\la_n)=(2c_1-e_1, 4c_2-e_2, \ldots, 2nc_n-e_n)\in L_n\]
given by Theorem~\ref{newCtoL}, in which $\sigma=(\sigma_1, \ldots, \sigma_n) \in C_n$, and $e=(e_1, \ldots, e_n) = \Psi(\sigma)$.

\subsection{The $\inv_{\tldC}$ statistic}\

\noindent
The Coxeter length of an element $w\in\tldC_n$ is its number of class inversions.  This corresponds to the sum of the parts of $\la$, as shown in Theorem~\ref{newCtoL}.
 \[\inv_{\tldC}(w) = \abs{\la}.\]
 
\subsection{The $\neg$ statistic}\

\noindent
We define $\neg(w)$ for $w\in\tldC_n/C_n$ to be $\neg(\sigma)$, the number of negative signs in its corresponding signed permutation $\sigma$.  Equivalently, $\neg(w)$ is the number of values in its window whose value modulo $N$ is greater than $n+1$.  This can be calculated directly from ceiling functions applied to~$\lam$.
\begin{lemma}
  \label{neg}
$$\neg(w) = 2 \abs{\ceil{\la}^{(2,4,\ldots,2n)}} - \abs{\ceil{\la}^{(1,2,\ldots,n)}}.$$ 
 \end{lemma}
 \begin{proof}
 From the definition of $\Psi$ we know that 
$ \neg(w) = \neg(\sigma) = \sum_{i=1}^n \floor{\frac{e_i}{i}}$.  So,
 \begin{eqnarray*}
 \neg(w) &=& \sum_{i=1}^n \floor{\frac{e_i}{i}}\\
 & = & \sum_{i=1}^n\floor{\frac{2i \ceil{\frac{\la_i}{2i}}-\la_i}{i}}\\
 & = & \sum_{i=1}^n  2\ceil{\frac{\la_i}{2i}} -  \sum_{i=1}^n \ceil{\frac{\la_i}{i}}\\
 & = & 2 \abs{\ceil{\la}^{(2,4,\ldots,2n)}} - \abs{\ceil{\la}^{(1,2,\ldots,n)}}.
 \end{eqnarray*}
 \end{proof}
In Lemma~\ref{lem:odd}, we show that this $\neg$ statistic equals the unexpected $o(\ceil{\la})$ statistic that occurs in the work of Bousquet-M\'elou and Eriksson in \cite{BME3}.

\subsection{The $\alpha$ and $\beta$ statistics}\

\noindent
The $\alpha$ and $\beta$ statistics on elements $w\in\tldC_n/C_n$ originally appeared in a refinement of Bott's formula.  The statistics $\alpha(w)$ and $\beta(w)$ are the number of occurrences of the generators $s_0$ and $s_n$ in any reduced word for $w$.  Bousquet-M\'elou and Eriksson translated them into the language of lecture hall partitions in their combinatorial proof of a refinement of Bott's formula \cite{BME3}.  We present a natural way in which the $\beta$ statistic appears.
\begin{lemma}
 \label{beta}
 $$\beta(w) = \sum_{i=1}^n c_i = \abs{\ceil{\la}^{(2,4,\ldots,2n)}}.$$
\end{lemma}
\begin{proof}
  The application of an $s_n$ generator as an ascent to a window takes the entry $w_j=c_jN+n$ that is equal to $n$ modulo $N$ and replaces it by $w_j+2=(c_j+1)N-n$, equal to $-n$ modulo $N$, in effect increasing $c_j$ by $1$.  From this we conclude that the number of $s_n$ generators is $\sum_{i=1}^n c_i$.
\end{proof}
The $\alpha$ statistic satisfies
 \begin{lemma}
 \label{alpha}
 $$ \alpha(w)= \sum_{i=1}^n c_i \ - \neg(w)=   \abs{\ceil{\la}^{(1,2,\ldots,n)}} -  \abs{\ceil{\la}^{(2,4,\ldots,2n)}}.$$
 \end{lemma}
 \begin{proof}
The application of an $s_0$ generator as an ascent to a window takes the entry $w_j=c_jN-1$ which is equal to $-1$ modulo $N$ and replaces it by $w_j+2=c_jN+1$, equal to $+1$ modulo $N$.  In particular, the number of $s_0$ generators is the number of times a negative residue modulo $N$ has become a positive residue modulo $N$.  Furthermore, recognize that every application of an $s_n$ generator replaces one positive residue modulo $N$ by a negative residue modulo $N$.  Since the identity permutation has all residues positive, and the number of $s_0$ applications would need to equal the number of $s_n$ applications in order to have all residues positive, then 
$\alpha(w)+\neg(w)=\beta(w)$.  Applying Lemmas~\ref{neg} and \ref{beta} completes the proof.
 \end{proof}
 
Putting together Lemmas~\ref{beta} and \ref{alpha} gives the following.
\begin{observation} The statistic  $\abs{\ceil{\la}^{(1,2,\ldots,n)}}$ on a lecture hall partition equals $\alpha(w) + \beta(w)$ for the corresponding element of $\tldC_n/C_n$.
\end{observation}

Other statistics have appeared in other related works.  In his thesis, Mongelli \cite{Mongelli} defines $\amaj(w)$ and $\ades(w)$ statistics on elements $w\in \tldC_n/C_n$.  (Mongelli's affine major statistic $\amaj$ is not the same as the Savage-Schuster ascent major statistic $\amaj$ presented later.)  Mongelli's statistics are equal to the statistics $\alpha(w)$ and $\beta(w)$, respectively.  Bousquet-M\'elou and Eriksson in \cite{BME3} define a statistic for the number of odd parts in $\ceil{\la}^{(1,2,\ldots,n)}$, which has a simple interpretation in terms of $\alpha(\la)$ and $\beta(\la)$.
 
 \begin{lemma}
\label{lem:odd}
$$ o(\lceil\la\rceil) = \neg(\sigma)= \beta(w)-\alpha(w).$$ 
\end{lemma}

\subsection{The $\max$ and $\last$ statistics}\

\noindent
Two statistics pertaining to the largest element of $w\in\tldC_n/C_n$ (and therefore $\lam\in L_n$) are $\last$ and $\max$.  Define $\last(w)=c_n$, which corresponds to the statistic 
\(
\ceil{\frac{\la_n}{2n}}
\) under the bijection of Theorem~\ref{newCtoL}.  We also wish to define a statistic $\max$ on $w\in\tldC_n/C_n$ that corresponds to the statistic $\la_n$ on lecture hall partitions.  
\begin{lemma}
\label{lem:max}
Under the definition
 \[
 \max(w)= w_n-\floor{\frac{w_n}{n+1}}-n
 \]
we have $\max(w)=\la_n$ when $w$ is the element of the parabolic quotient $\tldC_n/C_n$ that corresponds to the lecture hall partition $\la$.
\end{lemma}
\begin{proof}
Bradford et al \cite{Brant} approach lecture hall partitions through the use of the abacus model developed by Hanusa and Jones \cite{HJ12}, which in turn is a visualization of the work of \cite{EE} and \cite{BB} on representations of affine Coxeter groups permutations of $\mathbb{Z}$.  Under the conventions of \cite{Brant}, they prove that the value of the highest-numbered bead is exactly the largest part of the corresponding lecture hall partition.  

In order to convert our window notation to their set of lowest beads, we must take into account two differences in convention.  First, they choose $N=2n$ while we choose $N=2n+2$ by omitting runners $0$ and $n+1$.  Because of this, we must rescale every entry in the window by the ratio $\frac{n}{n+1}$ by subtracting $\big\lfloor{\frac{w_n}{n+1}}\big\rfloor$.  Second, our identity element is $[1,2,\hdots n]$ while their identity element chooses lowest beads $\{ -n+1,-n+2,\hdots 0\}$, which requires an additional adjustment of $-n$.  
\end{proof}

A related observation will be needed in Section~\ref{sec:smaller}.
\begin{observation}
	\label{ob:max}
   A lecture hall partition with largest part $\la_n=2tn$ corresponds to an element of the parabolic quotient $\tldC_n/C_n$ with largest entry $w_n=(n+1)(2t+1)$.	
\end{observation}

\subsection{The $\amaj$ and $\lhp$ statistics}\

\noindent
The following statistics on inversion sequences allow us to relate them to lecture hall partitions via Ehrhart theory.

We have already defined $\Asc^{(\s)}(e)$ and $\asc^{(\s)}(e)$ for $e \in \I_n^{(\s)}$.  We use two other statistics from \cite{SS} defined for a sequence $\s$: the ascent major index
$\amaj$, which is like a comajor index for the $\asc$ statistic, and the $\lhp$ statistic, which is inherited from the weight of a lecture hall partition. 
They are defined as:
\[
\amaj^{(\s)}(e) = \sum_{i \in \Asc^{(\s)}(e)} (n-i);
\]
\[
\lhp^{(\s)}(e) =  - \abs{e} + \sum_{i \in \Asc^{(\s)}(e)} (s_{i+1} + \ldots + s_n).
\]

Let $\comaj$ and $\lhp_C$ denote the type C version of the statistics, i.e.,  $\amaj^{(2,4, \ldots, 2n)}$ and 
$\lhp^{(2,4, \ldots, 2n)}$, respectively.  Then for $\sigma \in C_n$,
\[
\comaj(\sigma) = \sum_{i \in \Des(\sigma)} (n-i);
\]
and
\[
\lhp_C(\sigma) =  -\inv_C(\sigma) + \sum_{i \in \Des(\sigma)} (2(i+1) + \ldots + 2n).
\]

The statistic  $\lhp_C$ is {\em quadratic} in the sense that  its summands are quadratic functions of descent positions $i$.  In contrast, the summands of $\comaj$   are linear functions of descent positions $i$;  for the statistic $\des$, the summands are constant functions of $i$.

We will use these corresponding statistics on signed permutations to translate results from lecture hall partitions and inversion sequences to $C_n$ and $\tldC_n/C_n$.  We show in Section 4 that  $\lhp_C$ is the natural statistic to define  on $C_n$ so that the distribution of $\lhp_C$ over $C_n$ inherits a nice product-form generating function from Bott's formula for $\tldC_n/C_n$.  The proof relies on an Ehrhart theory result for lecture hall partitions together with the bijection $\tldC_n/C_n \rightarrow L_n$.

\newpage
\section{Translating results about lecture hall partitions into results about $\tldC_n/C_n$}
\label{sec:results}

We can use this new insight on the bijection to translate results between lecture hall partitions and elements of the parabolic quotient $\tldC_n/C_n$.
We use the standard $q$-notation, defining the $q$-bracket 
\[[k]_q=(1+q+\cdots+q^{k-1}),\] 
the $q$-factorial 
\[[k]_q!=\prod_{i=1}^k[i]_q,\] 
the $q$-binomial coefficients   
\[\qbinom{n}{k}_q=
\frac{[n]_q!}{[k]_q![n-k]_q!},\] and the $q$-Pochhammer symbols \[(a;q)_n=\prod_{k=0}^{n-1}(1-aq^k).\]

\subsection{Truncated coset representatives and Bott's Formula}\

Bousquet-M\'elou  and Eriksson  observed in  \cite{BME} that Bott's formula  \cite{Bott} for the Poincare series of  $\tldC_n$ is equivalent to:
\[
\sum_{w \in \tldC_n/C_n } q^{\ell(w)} = \prod_{i=1}^n \frac{1}{1-q^{2i-1}}.
\]
As noted earlier, $\ell(w) = \inv_{\tldC}(w).$  They used Bott's formula and their correspondence between lecture hall partitions and class inversion vectors for elements of $\tldC_n/C_n$ to give the first proof of the
{\em lecture hall theorem}:
\begin{theorem}[Lecture Hall Theorem, \cite{BME}]
\begin{equation*}\label{eq:lht}
\sum_{\la \in L_n} q^{|\la|} = \prod_{i=1}^n \frac{1}{1-q^{2i-1}}.
\end{equation*}
\end{theorem}
In a subsequent paper \cite{BME3}, Bousquet-M\'elou and Eriksson proved a refinement of the Lecture Hall Theorem corresponding to a refinement of Bott's formula.  Following their notation,  let
\[
\tldC_n/C_n (q,a,b) = \sum_{w \in \tldC_n/C_n } q^{\inv_{\tldC}(w)}a^{\alpha(w)}b^{\beta(w)}.
\] 
They noted that
\begin{equation}
\label{eq:Bott}
\tldC_n/C_n (q,a,b) = \prod_{i=1}^n  \frac{
1+bq^i}
{1-abq^{n+i}}.
\end{equation}
follows from a refinement of Bott's formula for the Poincar{\'e} series of $\tldC_n$, due to Macdonald, and is proved combinatorially by Eriksson and Eriksson in \cite{EE}.  Bousquet-M\'elou and Eriksson proved the following as a consequence of Equation~\eqref{eq:Bott} and also gave an independent combinatorial proof.
\begin{theorem}[\cite{BME3}]
\label{refined}
\[
\sum_{\la \in L_n} q^{\abs{\la}}u^{\abs{\ceil{\la}}}v^{o(\ceil{\la})} =  \prod_{i=1}^n  \frac{1+uvq^i}{1-u^2q^{n+i}}.
\]
\end{theorem}

We can use our correspondence from Section~\ref{sec:proof} to give a formula for a refinement of Bott's formula.  The so-called truncated lecture hall partitions $L_{n,k}$ are those in which  a specified number of parts are required to be zero:
\[
L_{n,k} = \left\{(\lam_1,\hdots,\lam_n) \in \integers^n: 0 < \frac{\lam_{n-k+1}}{n-k+1} \leq \frac{\lam_{n-k+2}}{n-k+2}\leq \cdots \leq \frac{\lam_n}{n}\right\}.
\]

\begin{lemma}
Under the bijection of Section~\ref{sec:proof}, the truncated lecture hall partitions $L_{n,k}$ correspond to the set $T_{n,k}$ of truncated minimal length coset representatives of $\tldC_n/C_n$ defined by
\[
T_{n,k} = \left \{w=[w_1,\hdots,w_n] \in \tldC_n/C_n \ : \ w_{n-k} \leq n \textup{ and } w_{n-k+1}>n \right \}.
\]
\end{lemma}

\begin{proof}
Under the bijection of Section~\ref{sec:proof}, the elements of $T_{n,k}$ are of the form 
\([c_1N+\sigma_1, \ldots, c_nN+\sigma_n]\) where \(c_1 = c_2 = \cdots = c_{n-k} = 0\) and $c_{n-k+1}>0$.  Since the elements of the window are in increasing order, the lemma follows.
\end{proof}

The truncated lecture hall partitions were enumerated in \cite{CS} by the statistics $|\la|$, $\big\lvert\lceil\la\rceil\big\rvert$, and $o(\lceil\la\rceil)$.  Using our thesaurus from Section~\ref{sec:stats}, we have the following refinement of Bott's formula to truncated minimal length coset representatives of $\tldC_n/C_n$.

\begin{theorem}
\[
\sum_{w \in T_{n,k}} q^{\inv_{\tldC}(w)} a^{\alpha(w)}b^{\beta(w)}= b^k q^{k+1 \choose 2} \qbinom{n}{k}_q 
\frac{(-aq^{n-k+1};q)_k}{(abq^{2n-k+1};q)_k}
\]
\end{theorem}

\subsection{Odd and even window entries}\

In \cite{BME}, Bousquet-M\'elou and Eriksson gave a proof of the Lecture Hall Theorem independent of Bott's formula.
It involved counting separately the weights 
\[
|\la|_o= \la_n + \la_{n-2} + \ldots
\]
and 
\[
|\la|_e= \la_{n-1} + \la_{n-3} + \ldots
\]
of a lecture hall partition $\la \in L_n$.  Indeed, they proved the following refinement of the lecture hall theorem:
\[
\sum_{\la \in L_n}  x^{|\la|_o}y^{|\la|_e} =  \prod_{i=1}^n \frac{1}{1-x^iy^{i-1}}.
\]
Using the correspondence between $L_n$ and $\tldC_n/C_n$ we can use this to state a different refinenent of Bott's formula in which, for $1 \leq i \leq n$, the class inversions with $w_i$  for  $w=[w_1, \ldots, w_n] \in \tldC_n/C_n$ are counted separately, depending on whether $i \in \{n, n-2, \ldots\}$ or $i \in \{n-1, n-3, \ldots \}$.

For $w \in \tldC_n/C_n$, let $$|w|_o= I_n(w) + I_{n-2}(w) + \ldots$$ and let  $$|w|_e= I_{n-1}(w) + I_{n-3}(w) + \ldots.$$
The correspondence gives us the following.
\begin{theorem}
\[
\sum_{w \in \tldC_n/C_n} x^{|w|_o}y^{|w|_e} = \prod_{i=1}^n \frac{1}{1-x^iy^{i-1}}.
\]
\end{theorem}
As in the preceding subsection, our tools allow us to exhibit a new further refinement of Bott's formula to truncated lecture hall partitions by reinterpreting Theorem~3 from \cite{CS}.  The corresponding result for $\tldC_n/C_n$ becomes:
\begin{theorem} For $n\geq k\geq 0$,
	\[\sum_{w \in T_{n,k}} x^{|w|_o}y^{|w|_e}=\frac{\big(x^{\lfloor k/2\rfloor+1}y^{\lceil k/2\rceil}\big)\displaystyle\qbinom{n-\lceil k/2\rceil}{\lfloor k/2\rfloor}_{xy}}{(x;xy)_{\lceil k/2 \rceil}\big(x^ny^{n-1};(xy)^{-1}\big)_{\lfloor k/2 \rfloor}}.\]
\end{theorem}

\subsection{Windows with smaller entries}\label{sec:smaller}\

\noindent
How many $\la \in L_n$ have largest part $\leq t$?  This is the lecture hall analog of counting partitions in an $n$ by $t$ box.  It has a nice answer if $t$ is expressed in the form $jn + k$.

\begin{theorem}[\cite{CLS}]
For $n \geq 0$, $j \geq 0$, and $0 \leq k \leq n$, the number of lecture hall partitions in $L_n$ satisfying $\la_n\leq  jn + k$ is
$(j+1)^{n-k}(j+2)^k$.
\end{theorem}
As a consequence we have the following.
\begin{corollary}
The number of lecture hall partitions $\la$ with $\la_n \leq 2tn$ is $(2t+1)^n$.
\end{corollary}

By Observation~\ref{ob:max}, we define a corresponding set $S_{n,t}$ of these smaller windows by 
\[
S_{n,t} = \big\{w \in \tldC_n/C_n \ : \ w_n \leq (2t+1)(n+1) \big\},
\]
and have the following corollary.

\begin{corollary}\label{cor:snt}
\[
|S_{n,t}|=(2t+1)^n.
\]
\end{corollary}

We can refine Corollary~\ref{cor:snt} to count via the statistics $\alpha$ and $\beta$.  To do this, we will apply Equation~3.6 from \cite{CLS15} which says
\[
\sum_{\la \in L_n, \la_n \leq 2t}u^{|\ceil{\la}|}v^{o(\la)} =
\left (
\frac{(1+uv-u^{2t+1}-u^{2t+2}}{1-u^2}
\right )^n.
\]

Under our correspondence, it follows that
\begin{theorem}\label{thm:smallwindows}
\[
\sum_{w \in S_{n,t}} a^{\alpha(w)}b^{\beta(w) }=
\left (
\frac{(1+b-a^tb^{t+1}(1+a)}{1-ab}
\right )^n.
\]
\end{theorem}
Note that setting $a=b$ gives
\begin{eqnarray*}
\sum_{w \in S_{n,t}} a^{\alpha(w)+\beta(w) } & = & 
\left (
\frac{(1+a)(1-a^{2t+1})}{1-a^2}
\right )^n\\
& = & [2t+1]_a^n.
\end{eqnarray*}

This proof of Theorem~\ref{thm:smallwindows} makes use of our correspondence.  An alternate direct proof chooses independently for all $i$ from $1$ to $n$ which of the $2t+1$ values of $i$ or $-i$ modulo $2n+2$ appears in the window.

\subsection{A quadratic statistic for $C_n$}\

\noindent
Ehrhart theory was applied in \cite{SS} to give a further connection between lecture hall partitions and signed permutations.
The following is a specialization of the main result of that paper (Theorem~6).
\begin{theorem}[\cite{SS}]
\[
\sum_{\la \in L_n} u^{\abs{\ceil{\la}_{2n}}} q^{\abs{\la}}= \frac{
\sum_{\sigma \in C_n} u^{\comaj(\sigma)} q^{\lhp_C(\sigma)} }
{\prod_{i=0}^{n-1}(1-u^{n-i}q^{2(i+1)+ \cdots + 2n})}.
\]
\end{theorem}
Our correspondence between $L_n$ and and $\tldC_n/C_n$ gives
\begin{equation}
\label{eq:parabolicEhrhart}
\sum_{w \in \tldC_n/C_n} u^{\beta(w)} q^{\inv_{\tldC}(w)}=  \frac{
\sum_{\sigma \in C_n} u^{\comaj(\sigma)} q^{\lhp_C(\sigma)} }
{\prod_{i=0}^{n-1}(1-u^{n-i}q^{2(i+1)+ \cdots + 2n})}.
\end{equation}
Applying the refinement of Bott's formula from Equation~\eqref{eq:Bott} when $a=1$ to Equation~\eqref{eq:parabolicEhrhart} gives the joint distribution of $\comaj$ and $\lhp$ on $C_n$.
\begin{corollary}
\begin{equation}
\label{eq:comajlhp}
		 \sum_{\sigma \in C_n} u^{\comaj(\sigma)} q^{\lhp_C(\sigma)}  = \prod_{i=1}^n  \frac{1+uq^i}{1-uq^{n+i}} 
	\prod_{i=0}^{n-1}(1-u^{n-i}q^{2(i+1)+ \cdots + 2n}).
\end{equation}
\end{corollary}

Equation~\eqref{eq:comajlhp} is a refinement of the well-known distribution
\[
 \sum_{\sigma \in C_n} u^{\comaj(\sigma)} = (1+u)^n\prod_{i=1}^n\, [i]_u.
 \]
If instead we specialize Equation~\eqref{eq:comajlhp} to when $u=1$, we arrive at an interesting distribution for the quadratic statistic $\lhp_C$ on $C_n$, which appears to be new.

\begin{corollary}\label{cor:quadlhp}
	\[
	\sum_{\sigma \in C_n} q^{\lhp_C(\sigma)}  = \prod_{k=1}^{n} [2k]_{q^{2(n-k)+1}}
	.\]
\end{corollary}

\begin{proof}
Specializing Equation~\eqref{eq:comajlhp} to when $u=1$ and simplifying gives
\[\begin{aligned}
\sum_{\sigma \in C_n} q^{\lhp_C(\sigma)}  &= \prod_{i=1}^n  \frac{(1+q^i)(1-q^{2i+ \cdots + 2n})}{1-q^{n+i}} \\
&= \prod_{i=1}^n  \frac{(1-q^{2i})(1-q^{2i+ \cdots + 2n})}{(1-q^i)(1-q^{n+i})} \\
&= \prod_{i=1}^n  \frac{(1-q^{2i+ \cdots + 2n})}{(1-q^{2i-1})} \\
&= \prod_{k=1}^{\lceil n/2 \rceil}  \frac{(1-q^{2(2k-1)(n-k+1)})}{(1-q^{2k-1})} \prod_{k=1}^{\lfloor n/2 \rfloor}  \frac{(1-q^{2k(2(n-k)+1)})}{(1-q^{2(n-k)+1})} \\
&= \prod_{k=1}^{\lceil n/2 \rceil} [2(n-k+1)]_{q^{2k-1}}  \prod_{k=1}^{\lfloor n/2 \rfloor} [2k]_{q^{2(n-k)+1}} \\
&= \prod_{k=n-\lfloor n/2 \rfloor}^{n} [2k]_{q^{2(n-k+1)-1}}  \prod_{k=1}^{\lfloor n/2 \rfloor} [2k]_{q^{2(n-k)+1}} \\
&= \prod_{k=1}^{n} [2k]_{q^{2(n-k)+1}} \\
\end{aligned}
\]
The fourth equality follows because the sum $i+\cdots+n$ has two nice formulas depending on whether there are an even or an odd number of terms.  If there are $2k-1$ terms, then the sum is $(2k-1)$ times the mean of the terms $(n-k+1)$.  If there are $2k$ terms, then the sum is $k$ times the sum of the first and the last term, $(n+n-2k+1)$.  
\end{proof}

Corollary~\ref{cor:quadlhp} can be viewed as a type~C analog of the following result from \cite{SS}.

\begin{theorem}[Corollary~5 of \cite{SS}]
	\[
	\sum_{\pi \in S_n} q^{\lhp(\pi)}  = \prod_{k=1}^{n} [k]_{q^{2(n-k)+1}}
	.\]
\end{theorem}

\section*{Acknowlegements}

	Thanks to the Simons Foundation for a grant to the second author which supported the travel of the first author for collaboration. The first author also gratefully acknowledges support from PSC-CUNY Research Award TRADA-47-191.

\bibliographystyle{amsalpha}
\bibliography{LHAC.0617}

\end{document}